\documentclass[12pt]{amsart}

\usepackage{amssymb}

\usepackage[all]{xy}

\makeatletter
\@namedef{subjclassname@2020}{%
  \textup{2020} Mathematics Subject Classification}
\makeatother

\usepackage[T1]{fontenc}

\newtheorem{theorem}{Theorem}[section]
\newtheorem{corollary}[theorem]{Corollary}
\newtheorem{lemma}[theorem]{Lemma}
\newtheorem{proposition}[theorem]{Proposition}

\theoremstyle{definition}
\newtheorem{definition}[theorem]{Definition}

\newtheorem{example}[theorem]{Example}

\numberwithin{equation}{section}

\DeclareMathOperator{\Hom}{\text{Hom}}

\DeclareMathOperator{\md}{\operatorname{mod}}
\DeclareMathOperator{\Ker}{\text{Ker}}
\DeclareMathOperator{\Aa}{\mathcal{A}}

\DeclareMathOperator{\pd}{\operatorname{pd}}
\DeclareMathOperator{\id}{\operatorname{id}}
\DeclareMathOperator{\gd}{\text{gl.dim}}
\DeclareMathOperator{\fd}{\text{fin.dim}}

\DeclareMathOperator{\supp}{\operatorname{supp}}

\newcommand{\lra}{\longrightarrow}

\newcommand{\benu}{\begin{enumerate}}
\newcommand{\enu}{\end{enumerate}}

\newcommand{\Co}{{\mathcal C}}

\newcommand{\dCo}{{\mathcal C}^*}

\begin{document}

\baselineskip=17pt


\title[Homological Invariants]{Homological invariants of generalized bound path algebras}

\author[Chust]{Viktor Chust}
\address{(Viktor Chust) Institute of Mathematics and Statistics - University of São Paulo, São Paulo, Brazil}
\email{viktorchust.math@gmail.com}

\author[Coelho]{Fl\'avio U. Coelho}
\address{(Flávio U. Coelho) Institute of Mathematics and Statistics - University of São Paulo, São Paulo, Brazil}
\email{fucoelho@ime.usp.br}

\date{}

\begin{abstract}
We study some homological invariants of a given generalized bound path algebra in terms of those of the algebras used in its construction. We discuss the particular case where the algebra is a generalized path algebra and give conditions for those algebras to be shod or quasitilted.
\end{abstract}



\maketitle

\section{Introduction}

An important result in the representation theory of algebras states that a finite dimensional basic $k$-algebra $A$, where $k$ is an algebraically closed field, is isomorphic to a quotient of a path algebra $kQ_A/I_A$ where $Q_A$ is a finite quiver and $I_A$ is an admissible ideal (see below for details). This allows us to describe the finitely generated $A$-modules in terms of the representations of the corresponding quiver $Q_A$, a  relation which proves to be of essential help in the theory.

In order to generalize such a construction, Coelho-Liu introduced in \cite{CLiu} the notion of generalized path algebras (gp-algebras for short). Instead of assigning the field $k$ to each vertex $i$ of the quiver $Q$ as in the traditional construction of path algebras $kQ$, it is assigned a finite dimensional $k$-algebra. This was further generalized by us in the article \cite{CC1} where we consider also some quotients of the gp-algebras. Specifically, let  $\Gamma$ denote a quiver and  $\Aa=\{A_i : i \in \Gamma_0\}$ denote a family of basic finite dimensional $k$-algebras  indexed by the set $\Gamma_0$ of the vertices of $\Gamma$. Consider also a set of relations $I$ on the paths of $\Gamma$ which generates an admissible ideal of $k\Gamma$. To such a data we have considered (see \cite{CC1}) the generalized bound path algebra $\Lambda = k(\Gamma,\Aa,I)$ (gbp-algebra for short) with a natural multiplication given not only by the concatenation of paths of the quiver but also by the multiplication of the algebras associated with the vertices of $\Gamma$, modulo the relations in $I$ (see below for details). 

The idea behind such a construction is to compare properties of the algebras in $\Aa$ and those for the algebra $\Lambda$. In the seminal work \cite{CLiu}, the focus was more ring theoretical, and, as mentioned, the authors only considered the case where $I=0$. We mention, for instance, \cite{FLi1,ICNLP,Kuls}, where such a particular case was also studied. 

In  \cite{CC1,CC2}, we have studied the case where $I$ is not necessarily zero, thus extending the description of the representations of the algebra $\Lambda$ given in \cite{ICNLP}.  Clearly, a path algebra $A$ can be realized as a generalized one in two trivial ways: the usual description as path algebra through its ordinary quiver $Q_A$ but also by considering a quiver with a single vertex and no arrows and assigning to it the whole algebra $A$. In \cite{CC1}, we discuss when there are, besides the above two, other possibilities of realizing a path algebra as a generalized one. This is important  because we can relate properties of a gp-algebra with those of the {\it smaller algebras} which appear in its definition. In \cite{CC2}, we studied the correspondence between modules over a gbp-algebra and representations of the corresponding quiver.

Here, our focus will be, using the description of the projective and injective modules from \cite{CC2}, to study some homological invariants of a given gbp-algebra in terms of those of the algebras used in its construction. This is done in Sections 3 and 4 after devoting Section 2 to  preliminary concepts needed along the paper. The particular case of gp-algebras is discussed in Section 5 where we prove, for instance, that the global dimension of a gp-algebra is the maximum between one and the global dimension of the algebras assigned to each vertex (Theorem 5.1). Also, we provide a sufficient condition for a gp-algebra to belong to classes of algebras which can be defined using these invariants, such as {\it shod} or {\it quasitilted} algebras (see \cite{CLa1,HRS}). 

\section{Preliminaries} 

Along this paper, $k$ will denote an algebraically closed field. For an algebra, we mean an associative and unitary basic finite dimensional $k$-algebra. Also, given an algebra $A$,  an $A$-module (or just a module) will be a finitely generated right module over $A$. We refer to \cite{AC,ARS} for unexplained details on Representation Theory.

\subsection{Path algebras}
\label{subsec:quivers}

A {\bf  quiver} $Q$  is given by a tuple $(Q_0, Q_1, s,e)$, where $Q_0$ is the set of {\bf vertices}, $Q_1$ is the set of {\bf arrows} and $s,e \colon Q_1 \lra Q_0$ are maps which indicate, for each arrow $\alpha \in Q_1$, the {\bf starting vertex} $s(\alpha)\in Q_0$ of $\alpha$ and the {\bf ending vertex} $e(\alpha) \in Q_0$ of $\alpha$.  A vertex  $i \in Q_0$ is called a {\bf source} (respectively a {\bf sink}) provided there are no arrows ending (or starting, respectively) at $i$. A {\bf path in $Q$ of length $n \geq 1$} is given by $\alpha_1 \cdots \alpha_n$, where for each $i = 1, \cdots, n-1$, $e(\alpha_i) = s(\alpha_{i+1})$. There are also {\bf paths of length zero} which are in one-to-one correspondence to the vertices of $Q$. 

We shall assume that all quivers are finite, that is, both sets $Q_0$ and $Q_1$ are finite.

Naturally, given a quiver $Q$ one can assign a {\bf path algebra} $kQ$ with a $k$-basis given by all paths of $Q$ and multiplication on that basis defined by concatenation. Even when $Q$ is finite, the corresponding algebra could not be finite dimensional. However, a well-known result established by Gabriel states that given an algebra $A$, there exists a finite quiver $Q$ and a set of relations on the paths of $Q$ which generates an admissible ideal $I$ such that $A \cong kQ/I$ (see \cite{AC} for details).

\subsection{Generalized bound path algebras (gbp-algebras)}

Let $\Gamma= (\Gamma_0,\Gamma_1,s,e)$ be a quiver and $\mathcal{A}=(A_i)_{i \in \Gamma_0}$ be a family of algebras indexed by $\Gamma_0$.   An \textbf{$\mathcal{A}$-path of length $n$} over $\Gamma$ is defined as follows: for $n = 0$, such a path is  an element of $\bigcup_{i \in \Gamma_0} A_i$, and for $n>0$, it  is a  sequence of the form
$$a_1 \beta_1 a_2 \ldots a_n \beta_n a_{n+1}$$
where $\beta_1 \ldots \beta_n$ is an ordinary path in the quiver $\Gamma$, $a_i \in A_{s(\beta_i)}$ if $i \leq n$, and $a_{n+1} \in A_{e(\beta_n)}$.  Denote by $k[\Gamma,\mathcal{A}]$ the $k$-vector space spanned by all $\mathcal{A}$-paths over $\Gamma$. 

Then we consider the quotient vector space $k(\Gamma,\mathcal{A})=k[\Gamma,\mathcal{A}]/V$, where $V$ is the subspace generated by all elements of the form
$$(a_1 \beta_1 \ldots \beta_{j-1}(a^1_j+ \ldots+a^m_j)\beta_j a_{j+1} \ldots \beta_n a_{n+1}) - \sum_{l=1}^m (a_1 \beta_1 \ldots \beta_{j-1} a_j^l \beta_j \ldots \beta_n a_{n+1})$$
or, for $\lambda \in k$,
$$(a_1 \beta_1 \ldots \beta_{j-1}( \lambda a_j) \beta_j a_{j+1} \ldots \beta_n a_{n+1})- \lambda\cdot (a_1 \beta_1 \ldots \beta_{j-1} a_j \beta_j a_{j+1} \ldots \beta_n a_{n+1}).$$

The space $k(\Gamma,\mathcal{A})$ has a naturally defined multiplication, induced by the multiplications of the algebras  $A_i$'s and the composition of the $\mathcal{A}$-paths. More explicitly, it is defined by linearity and the following rule:
$$(a_1 \beta_1 \ldots \beta_n a_{n+1})(b_1 \gamma_1 \ldots \gamma_m b_{m+1}) = a_1 \beta_1 \ldots \beta_n (a_{n+1} b_1) \gamma_1 \ldots \gamma_m b_{m+1}$$
if $e(\beta_n) = s(\gamma_1)$, and 
$$(a_1 \beta_1 \ldots \beta_n a_{n+1})(b_1 \gamma_1 \ldots \gamma_m b_{m+1}) = 0 $$
otherwise.

With this multiplication,  $k(\Gamma,\mathcal{A})$  is an associative algebra, and since we are assuming the quivers to be finite, it has also an identity element, which is equal to $\sum_{i \in \Gamma_0} 1_{A_i}$. Finally, it is easy to observe that $k(\Gamma,\Aa)$ is finite-dimensional over $k$ if and only if so are the algebras $A_i$ and if $\Gamma$ is acyclic. We call  $k(\Gamma,\mathcal{A})$  the   \textbf{generalized path algebra} ({\bf gp-algebra}) over $\Gamma$ and $\mathcal{A}$ (see \cite{CLiu}). In case $A_i = k$ for every $i \in \Gamma_0$, this construction gives the usual path algebra $k\Gamma$.

Using the result mentioned above in \ref{subsec:quivers},  for each $i \in \Gamma_0$, we fix a quiver $\Sigma_i$ such that $A_i \cong k\Sigma_i/ \Omega_i$ with $\Omega_i$  an admissible ideal of $k\Sigma_i$.

 Following \cite{CC1}, we shall now consider quotients of generalized path algebras by an ideal generated by relations. Namely, let $I$ be a finite set of relations over $\Gamma$ which generates an admissible ideal in $k\Gamma$.  Consider the ideal $(\Aa(I))$ generated by the following subset of  $k(\Gamma,\Aa)$:
\begin{align*}
\Aa(I)&= \left\{ \sum_{i = 1}^t \lambda_i  \beta_{i1} \overline{\gamma_{i1}} \beta_{i2} \ldots \overline{\gamma_{i(m_i-1)}} \beta_{im_i} : \right.\\
&\left. \sum_{i = 1}^t \lambda_i \beta_{i1} \ldots \beta_{im_i} \text{  is a relation in } I 
\text{ and }\gamma_{ij}\text{ is a path in }\Sigma_{e(\beta_{ij})} \right\}
\end{align*}

The quotient $\frac{k(\Gamma,\mathcal{A})}{(\Aa(I))}$ is said to be a \textbf{generalized bound path algebra (gbp-algebra)}. We may also write $\frac{k(\Gamma,\mathcal{A})}{(\Aa(I))}=k(\Gamma,\Aa,I)$. When the context is clear, we simply denote the set $\Aa(I)$  by $I$.

We  use the following notation in the sequel: $\Gamma$ is an acyclic quiver, $\Aa=\{A_i : i \in \Gamma_0\}$  denotes a family of basic finite dimensional algebras over an algebraically closed field $k$, and $I$ is a set of relations in $\Gamma$ generating an admissible ideal in the path algebra $k\Gamma$. By $\Lambda = k(\Gamma,\Aa,I)$, we denote the gbp-algebra obtained from these data. Also, $A_{\Aa}$ will denote the product algebra $\prod_{i \in \Gamma_0} A_i$.  We denote the identity element of the algebras $A_i$ by $1_i$ instead of $1_{A_i}$. Also, for an algebra $A$, we shall denote by mod$A$ the category of finitely generated right $A$-modules. 

\subsection{Representations}  \label{subsec:representations}

In \cite{CC2}, we have described the representations of a gbp-algebra, including those associated to projective and injective modules. We shall now recall the results needed in the sequel. 

\begin{definition} Let $\Lambda = k(\Gamma,\Aa,I)$ be a gbp-algebra. 

(a) A {\bf representation} of $\Lambda$ is given by 
$ ((M_i)_{i \in \Gamma_0},(M_{\alpha})_{\alpha \in \Gamma_1})$
where
\begin{enumerate}

\item[(i)] $M_i$ is an $A_i$-module, for each $i \in \Gamma_0$;
\item[(ii)] $M_{\alpha}: M_{s(\alpha)} \rightarrow M_{e(\alpha)}$ is a $k$-linear transformation, for each  arrow $\alpha \in \Gamma_1$ ; and 
\item[(iii)] whenever $\gamma = \sum_{i = 1}^t \lambda_t \alpha_{i1} \alpha_{i2} \ldots \alpha_{in_1}$ is a relation in $I$ where $\lambda_i \in k$ and $\alpha_{ij} \in \Gamma_1$, then 
$$\sum_{i = 1}^t \lambda_t M_{\alpha_{in_i}} \circ \overline{\gamma_{in_i}} \circ \ldots \circ M_{\alpha_{i2}} \circ \overline{\gamma_{i2}} \circ M_{\alpha_{i1}} = 0$$
for every choice of paths $\gamma_{ij}$ over $\Sigma_{s(\alpha_{ij})}$, with $1 \leq i \leq t$, $2 \leq j \leq n_i$.
\end{enumerate}

(b) We say that a representation $ ((M_i)_{i \in \Gamma_0},(M_{\alpha})_{\alpha \in \Gamma_1})$ of $\Lambda$ is {\bf finitely generated} if each of the $A_i$-modules $M_i $ is finitely generated. 

(c) Let $M = ((M_i)_{i \in \Gamma_0},(M_{\alpha})_{\alpha \in \Gamma_1})$ and $N = ((N_i)_{i \in \Gamma_0},(N_{\alpha})_{\alpha \in \Gamma_1})$ be representations of  $\Lambda$. A \textbf{morphism of representations} $f: M \rightarrow N$ is given by a tuple $f = (f_i)_{i \in \Gamma_0}$, such that, for every $i \in \Gamma_0$, $f_i: M_i \rightarrow N_i$ is a morphism of $A_i$-modules; and such that, for every arrow $\alpha: i \rightarrow j \in \Gamma_1$, it holds that $f_j M_{\alpha} = N_{\alpha} f_i$, that is, the following diagram comutes:
$$ \xymatrix{M_i \ar[r]^{M_{\alpha}} \ar[d]_{f_i} & M_j \ar[d]^{f_j}\\
N_i \ar[r]_{N_{\alpha}} & N_j} $$
\end{definition}

We shall denote by Rep$_k(\Gamma,\Aa,I)$ (or by rep$_k(\Gamma,\Aa,I)$) the category of the representations (or finitely generated representations, respectivelly) of the algebra $k(\Gamma,\Aa,I)$. 

\begin{theorem}[\cite{CC2}, see also \cite{ICNLP}]
\label{th:rep_icnlp}
There is a $k$-linear equivalence 
$$F: \operatorname{Rep}_k(\Gamma,\Aa,I) \rightarrow \operatorname{Mod} k(\Gamma,\Aa,I)$$
which restricts to an equivalence 
$$F: \operatorname{rep}_k(\Gamma,\Aa,I) \rightarrow \operatorname{mod} k(\Gamma,\Aa,I)$$
\end{theorem}

\subsection{Realizing an $A_i$-module as a $\Lambda$-module} Let $i \in \Gamma_0$ and let $M$ be an $A_i$-module. We consider three ways of realizing $M$ as a $\Lambda$-module:
\vspace{.3 cm}\\
{\bf A- Natural inclusion.} ${\mathcal{I}}(M) = ((M_j)_{j \in \Gamma_0},(\phi_{\alpha})_{\alpha \in \Gamma_1}) $ is the representation given by
$$M_j = 
\begin{cases}
M & \text{ if } j = i \\
0 & \text{ if } j \neq i
\end{cases} \ \ \ \mbox{ and } \ \ \ \phi_{\alpha} = 0 \hspace{1ex}\text{ for all }\alpha \in \Gamma_1. $$
By abuse of notation, we shall identify ${\mathcal{I}}(M) = M$, since these two have the same underlying space. 
\vspace{.3 cm}\\
{\bf B-Cones.} As observed in \cite{CLiu}, if $k(\Gamma,\Aa)$ is a gp-algebra, then it is a tensor algebra: if $A_{\Aa} = \prod_{i \in \Gamma_0} A_i$ is the product of the algebras in $\Aa$, then there is an $A_{\Aa}$-$A_{\Aa}$-bimodule $M_{\Aa}$ such that $k(\Gamma,\Aa) \cong T(A_{\Aa},M_{\Aa})$. 

Since $M$ is naturally an $A_{\Aa}$-module and there is a canonical map $A_{\Aa} \rightarrow \Lambda = k(\Gamma,\Aa)/I$, then, by extension of scalars, $M$ originates a $\Lambda$-module $\Co_i(M)$, which is called the \textbf{cone} over $M$.

We now recall the following results from \cite{CC2}:

\begin{proposition}
\label{prop:cone is exact}
\label{cone x projective}
Given $i \in \Gamma_0$, we have:
\begin{enumerate}
    \item The cone functor $\Co_i : \md A_i \rightarrow \md \Lambda$  is exact.

    \item If $P$ is a projective $A_i$-module, then $\Co_i(P)$ is a projective $\Lambda$-module.
\end{enumerate}
\end{proposition}

\vspace{.3 cm}
\par\noindent
{\bf C-Dual cones.} The \textbf{dual cone} over $M$ is given by $\dCo_i(M) \doteq D \Co_i D (M)$, where\break 
$D = \Hom_k(-,k)$ is the usual duality functor. A dual result of Proposition \ref{prop:cone is exact} for injective modules holds true (see \cite{CC2}).

We refer to \cite{CC2} for further details of the above constructions. 

\section{Homological dimensions}

Using the notations established above, we shall concentrate now in the comparison of some homological dimensions of $\Lambda$ with those in the algebras $A_i$, $i\in \Gamma_0$. Given an algebra $A$ and an $A$-module $M$, we denote by pd$_AM$ and by id$_AM$ the projective and the injective dimensions of $M$, respectively. Also, the global dimension of $A$ is denoted by gl.dim$A$.

\subsection{First case} We analyse the natural inclusion of $A_i$-modules in mod$\Lambda$. 

\begin{lemma}    \label{lem:easy_ineq_pd}
Let $i \in \Gamma_0$ and let $M$ be an $A_i$-module. Then 
\begin{enumerate}
\item[(a)] $\pd_{\Lambda} M \geq \pd_{A_i} M$.
\item[(b)]  if $i$ is a sink, then $\pd_{\Lambda} M = \pd_{A_i} M$.
\item[(c)]  $\id_{\Lambda} M \geq \id_{A_i} M$.
\item[(d)] if $i$ is a source, then  $\id_{\Lambda} M = \id_{A_i} M$.
\end{enumerate}
\end{lemma}

\begin{proof} We shall prove only (a) and (b) since the proofs of (c) and (d) are dual. \\
(a) There is nothing to show if pd$_\Lambda M = \infty$. So, assume $M$ has finite projective dimension $m$ over $\Lambda$ and let 
\begin{displaymath}
\xymatrix{
0 \ar[r] & P_m \ar[r] & \ldots \ar[r] & P_1 \ar[r] & P_0 \ar[r] & M \ar[r] & 0  }
\end{displaymath}
be a minimal projective $\Lambda$-resolution of $M$. Since a projective resolution is in particular an exact sequence, it yields an exact sequence of $A_i$-modules at the $i$-th component:
\begin{displaymath}
\xymatrix{
0 \ar[r] & (P_m)_i \ar[r] & \ldots \ar[r] & (P_1)_i \ar[r] & (P_0)_i \ar[r] & M \ar[r] & 0   }
\end{displaymath}
It follows from the description of the projective modules over $\Lambda$ (see \cite{CC2}, Subsection 5.1) that every component of a projective representation is projective (indeed, the $i$-th  component is a direct sum of indecomposable projective modules over $A_i$, copies of $A_i$, or zero modules). Thus the exact sequence above is a projective resolution of $M$ over $A_i$. This implies that $\pd_{A_i} M \leq m = \pd_{\Lambda} M$.\\
(b)  Because $i$ is a sink,  every projective resolution of $M$ over $A_i$ is easily seen to yield a projective resolution of $M$ over $\Lambda$ with the same length.
\end{proof}

The next result follow now easily. 

\begin{corollary} \label{cor:easy_ineq_gd}
$\gd \Lambda \geq \operatorname{max} \{\gd A_1, \ldots, \gd A_n \}$.
\end{corollary}

We shall see below examples of when equality in the above statement holds and when it does not.

\subsection{Cones and duals} The next result, which relates the projective and the injective dimensions of a module over $A_i$ with the corresponding dimension of its cone or its dual cone,  is a direct consequence of Proposition~\ref{prop:cone is exact} and its dual. 

\begin{lemma}
\label{lem:pd x cone}
Given $i \in \Gamma_0$ and $M$ an $A_i$-module, Then 
\begin{enumerate}
\item[(a)] $\pd_{A_i} M = \pd_{\Lambda} \Co_i(M)$.
\item[(b)]  $\id_{A_i} M = \id_{\Lambda} \dCo_i(M)$.
\end{enumerate}
\end{lemma}

\begin{proof} We shall prove only (a) since the proof of (b) is dual. Let 
\begin{displaymath}
\xymatrix{
0 \ar[r] & P_m \ar[r] & \ldots \ar[r] & P_1 \ar[r] & P_0 \ar[r] & M \ar[r] & 0
}
\end{displaymath}
be a minimal projective resolution of $M$ in $\md A_i$. Thus $m = \pd_{A_i} M$. Applying the functor $\Co_i$, we have
\begin{displaymath}
\xymatrix{
0 \ar[r] & \Co_i(P_m) \ar[r] & \ldots \ar[r] & \Co_i(P_1) \ar[r] & \Co_i(P_0) \ar[r] & \Co_i(M) \ar[r] & 0
}
\end{displaymath}
Because of Proposition~\ref{prop:cone is exact}, this sequence is exact. Moreover, also by Proposition~\ref{cone x projective}, every term except possibly for $\Co_i(M)$ is known to be projective. So this is a projective resolution in $\md \Lambda$, proving that $\pd_{\Lambda} \Co_i(M) \leq \pd_{A_i} M$. Since the $i$-th component of $\Co_i(M)$ is $M$, we know from Proposition~\ref{lem:easy_ineq_pd} that the inverse inequality also holds.
\end{proof}


\subsection{General case} Having studied the projective and injective dimensions of modules which are inclusion or cones of $A_i$-modules, we turn our attention to general $\Lambda$-represen\-ta\-tions. 

\begin{definition}
Let $M = ((M_i)_{i \in \Gamma_0},(\phi_{\alpha})_{\alpha \in \Gamma_1})$ be a representation over $k(\Gamma,\Aa,I)$. The \textbf{support} of $M$ is defined as the set of vertices $\supp M \doteq \{i \in \Gamma_0: M_i \neq 0\}$.
\end{definition}

\begin{proposition}
\label{prop:pd_submodule_quotient}
For a $\Lambda$-module  $M$,  
\benu
\item[(a)] $\pd_{\Lambda} M \leq \max_{j \in \supp M}\{\pd_{\Lambda} M_j\}$
\item[(b)] $\id_{\Lambda} M \leq \max_{j \in \supp M}\{\pd_{\Lambda} M_j\}$
\enu
\end{proposition}

\begin{proof} We shall prove only (a) since the proof of (b) is dual. \\
We proceed by induction on $|\supp M|$. If $|\supp M|=1$ then $M$ has only one non-zero component, say the $i$-th component, and it is clear that $\pd_{\Lambda} M = \pd_{\Lambda} M_i$. \\
Suppose $|\supp M| > 1$ and that the statement holds for representations whose support is smaller than that of $M$. Then, since $\Gamma$ is acyclic, there is at least one vertex $i \in \Gamma_0$ which is a source in the full subquiver determined by $\supp M$. We consider the following representations:
$$N = ((N_j)_{j \in \Gamma_0}, (\psi_{\alpha})_{\alpha \in \Gamma_1}) \ \ \mbox{ given by } \ \ 
N_i = M_i \text{, } N_j = 0 \text{ if } j \neq i \text{, } \psi_{\alpha} = 0  \mbox{; and} $$
$$T = ((T_j)_{j \in \Gamma_0}, (\rho_{\alpha})_{\alpha \in \Gamma_1})\ \ \mbox{ given by } \ \ 
T_i = 0 \text{, } T_j = M_j \text{ if } j \neq i \text{, } \rho_{\alpha} = \phi_{\alpha}|_{T_{s(\alpha)}}$$
Observe that, since the support of $N$ has size 1, it satisfies the relations in $I$. Also, it is easy to see that $T$ also satisfies these relations, because $M$ does.\\
We also consider two morphisms of representations $f=(f_j)_{j \in \Gamma_0}: T \rightarrow M$, and  $g=(g_j)_{j \in \Gamma_0}: M \rightarrow N$, given by:
$$f_j:T_j \rightarrow M_j \text{, } f_i = 0 \text{, } f_j = id_{M_j} \text{ if } j \neq i \mbox{; and}$$ 
$$g_j:M_j \rightarrow N_j \text{, } g_i = id_{M_i} \text{, } g_j = 0 \text{ if } j \neq i$$
It is directly verified that these are in fact morphisms of representations. Thus we have a short exact sequence of representations:
$$0 \rightarrow T \xrightarrow{f} M \xrightarrow{g} N \rightarrow 0$$
It is indeed exact because the $i$-th component is 
$$0 \rightarrow 0 \rightarrow M_i \xrightarrow{id_{M_i}} M_i \rightarrow 0$$
and for $j \neq i$, the $j$-component is 
$$0 \rightarrow M_j \xrightarrow{id_{M_j}} M_j \rightarrow 0 \rightarrow 0$$
and these are clearly exact sequences. \\ 
We obtain that $\pd_{\Lambda} M \leq \max\{\pd_{\Lambda} T, \pd_{\Lambda} N\}$. Note that $|\supp T| = n-1$ and $|\supp N|=1$. Therefore, by the induction hypothesis, 
$$\pd_{\Lambda} N = \pd_{\Lambda} N_i = \pd_{\Lambda} M_i$$
$$\pd_{\Lambda} T \leq \max_{j \in \supp T}\{\pd_{\Lambda} T_j\} = \max_{\substack{j \in \supp M \\ j \neq i}}\{\pd_{\Lambda} T_j\} = \max_{\substack{j \in \supp M \\ j \neq i}}\{\pd_{\Lambda} M_j\}$$ 
Assembling the pieces together, we conclude that $$\pd_{\Lambda} M \leq \max\{\pd_{\Lambda}N,\pd_{\Lambda}T\} \leq \max_{j \in \supp M}\{\pd_{\Lambda} M_j\},$$ 
as we wanted to prove.
\end{proof}

\section{Homological dimensions for gbp-algebras}

We shall prove in this section some general results involving gbp-algebras, leaving the particular case of gp-algebras for the next section. 
We will adopt the following notation from here on: if $i$ is a source vertex of $\Gamma$, then $\Gamma \setminus \{i\}$ shall denote the quiver obtained from $\Gamma$ by deleting the vertex $i$ and the arrows starting at $i$. Moreover, if $\Gamma$ is equipped with a set of relations $I$, $I \setminus \{i\}$ will be the set obtained from $I$ by excluding the relations starting at $i$. Also, since $\Gamma$ is acyclic, we can iterate this process and enumerate $\Gamma_0 = \{1,\ldots,n\}$ in such a way that  $i$ is a source vertex of $\Gamma \setminus \{1,\ldots,i-1\}$ for every $i$.

\begin{lemma}
\label{lem:main lemma}
Let $i \in \Gamma_0$, $M$ be an $A_i$-module, and let $(P,g)$ be its projective cover in mod $A_i$. 
Then there is an exact sequence of $\Lambda$-modules:
\begin{displaymath}
\xymatrix{
0 \ar[r]& \Co_i(\Ker g) \oplus L \ar[r]& \Co_i(P) \ar[r]& M \ar[r]& 0}
\end{displaymath}
where $L$ is a $\Lambda$-module  such that 
$$\supp L \subseteq \{j \in \Gamma_0: j \neq i \text { and there is a path }i \rightsquigarrow j\}$$
Moreover, 
\begin{enumerate}
\item[(a)] $L_j$ is free for every vertex $j$, and
\item[(b)] If $i \in \Gamma_0$ is such that $I \setminus \{1,\ldots,i\} = I \setminus \{1,\ldots,i-1\}$, then $L$ is projective over $\Lambda$.
\end{enumerate}
\end{lemma}

\begin{proof}

(a) It follows from \cite{CC2} (Proposition 5 and Remark 5) 
that $(\Co_i(P))_i = P$. So,  we can define a morphism of representations $g': \Co_i(P) \rightarrow M$ by establishing that $g'_i = g$ and that $g'_j = 0$ for $j \neq i$. We want to show that $\Ker g' = \Co_i(\Ker g) \oplus L$, where $L$ satisfies the conditions in the statement. \\
Let $\{p_1,\ldots,p_r\}$ be a $k$-basis of $\Ker g$ and complete it to a $k$-basis \break $\{p_1,\ldots,p_r,\ldots,p_s\}$ of $P$. Also let, for every $j \in \Gamma_0$, $\{a_1^j,\ldots,a_{n_j}^j\}$ be a $k$-basis of $A_j$. For a path $\gamma: i = l_0 \rightarrow l_1 \rightarrow \ldots \rightarrow l_t = j$ from  $i$ to  $j$ in $\Gamma$ denote
$$\theta_{\gamma,h,i_1,\ldots,i_t} = p_h \otimes \overline{\gamma_1 a_{i_1}^{e(\gamma_1)} \gamma_2 a_{i_2}^{e(\gamma_2)}\ldots \gamma_t a_{i_t}^j} \in \Ker g'$$
Remember that since $g'$ was defined as a morphism of representations, it corresponds to a morphism of $\Lambda$-modules, because of Theorem~\ref{th:rep_icnlp}. Therefore,
$$g'(\theta_{\gamma,h,i_1,\ldots,i_t}) = g'(p_h \otimes \overline{\gamma_1 a_{i_1}^{e(\gamma_1)} \gamma_2 a_{i_2}^{e(\gamma_2)}\ldots \gamma_t a_{i_t}^j}) = g(p_h)$$
So $\theta_{\gamma,h,i_1,\ldots,i_t} \notin \Ker g'$ if and only if $\gamma$ is the zero-length path $\epsilon_i$ and $r < h \leq s$. Thus we can write
\begin{align*}
\Ker g' &= (\theta_{\epsilon_i,h}:1 \leq h \leq r) + (\theta_{\gamma,h,i_1,\ldots,i_t}: l(\gamma) > 0) \\
&= (\theta_{\gamma,h,i_1,\ldots,i_t}: 1 \leq h \leq r) \oplus (\theta_{\gamma,h,i_1,\ldots,i_t}: l(\gamma) > 0 \text{ and } r < h \leq s) \\
&= \Co_i(\Ker g) \oplus L
\end{align*}
where $L \doteq (\theta_{\gamma,h,i_1,\ldots,i_t}: l(\gamma) > 0 \text{ and } r < h \leq s)$. Since the generators of $L$ involve only paths of length strictly greater than zero, the only components of $L$ that are non-zero are the ones over the successors of $i$, except for $i$ itself. Therefore the condition about the support of $L$ in the statement is satisfied. It remains to prove the other two assertions in the statement. \\
To prove (a), fix $j \in \Gamma_0$. If $j = i$ or if $j$ is not a successor of $i$, then $L_j = 0$, so we may suppose this is not the case. Again using the equivalence given by Theorem~\ref{th:rep_icnlp},
\begin{align*}
L_j &= L\cdot {1_j} = (\theta_{\gamma,h,i_1,\ldots,i_t}: \gamma: i \rightsquigarrow j \text{ and } r < h \leq s) \\
&= (p_h \otimes \overline{\gamma_1 a_{i_1}^{e(\gamma_1)} \gamma_2 a_{i_2}^{e(\gamma_2)}\ldots \gamma_t a_{i_t}^j}:\gamma: i \rightsquigarrow j \text{ and } r < h \leq s)
\end{align*}
So $L_j$ is isomorphic to the free $A_j$-module whose basis is the set of all possible elements $p_h \otimes \overline{\gamma_1 a_{i_1}^{e(\gamma_1)} \gamma_2 a_{i_2}^{e(\gamma_2)}\ldots \gamma_t}$. In particular, $L_j$ is free over $A_j$, and this proves (a).\\
(b) Assume that  $I \setminus \{1,\ldots,i\} = I \setminus \{1,\ldots,i-1\}$ and let $i^+$ denote the set of immediate successors of $i$. Since, by hypothesis, there are no relations starting at $i$, we can write:
\begin{align*}
L &\doteq (\theta_{\gamma,h,i_1,\ldots,i_t}: l(\gamma) > 0 \text{ and } r < h \leq s) \\
&= (p_h \otimes \overline{\gamma_1 a_{i_1}^{e(\gamma_1)} \gamma_2 a_{i_2}^{e(\gamma_2)}\ldots \gamma_t a_{i_t}^j}: l(\gamma) > 0 \text{ and } r < h \leq s) \\
&= (p_h \otimes \gamma_1 a_{i_1}^{e(\gamma_1)} \otimes \overline{ \gamma_2 a_{i_2}^{e(\gamma_2)}\ldots \gamma_t a_{i_t}^j}: l(\gamma) > 0 \text{ and } r < h \leq s) \\
&\cong \left( a_{i_1}^{e(\gamma_1)} \otimes \overline{ \gamma_2 a_{i_2}^{e(\gamma_2)}\ldots \gamma_t a_{i_t}^j}: l(\gamma) > 0 \right)^{s-r}\\
&\cong \bigoplus_{i' \in i^+} \Co_i(A_{i'})^{s-r}
\end{align*}
Since $A_{i'}$ is projective over $A_{i'}$, then $\Co_i(A_{i'})$ is projective over $\Lambda$ by Proposition~\ref{cone x projective}. We have thus shown that $L$ is isomorphic to a direct sum of projective $\Lambda$-modules, and therefore it is also projective, concluding the proof. 
\end{proof}

\subsection{A special kind of gbp-algebras} 
Before our next result, we need a further definition. For a vertex $j$ of $\Gamma$, denote by $S_j$ the simple $k\Gamma/I$-module over $j$.

\begin{definition}
A gbp-algebra $\Lambda$ is called {\bf terraced} provided  for every $i \in \Gamma_0$ such that 
$I \setminus \{1,\ldots,i\} \neq I \setminus \{1,\ldots,i-1\}$ (i.e., every time there are relations starting at $i$), one has  
$$\pd_{k\Gamma/I} S_i \geq \max \{\pd_{k\Gamma/I} S_j : j \text{ is a successor of }i\}+1.$$ 
\end{definition}

Observe that any gp-algebra (that is, when $I = 0$, which makes $k\Gamma$ hereditary) is  terraced.

\begin{theorem}
\label{th:pd_formula}
Let $\Lambda = k(\Gamma,\Aa,I)$ be a terraced gbp-algebra. 
Then, for every representation $M$ over $\Lambda$,
$$\pd_{\Lambda} M \leq \max_{i \in \supp M}\{\pd_{A_i}M_i,\pd_{k\Gamma/I}S_i\}$$
where $S_i$ denotes the simple $k\Gamma/I$-module associated with the vertex $i$.
\end{theorem}

\begin{proof}
The proof is done by induction. First, suppose $\supp M = \{n\}$. By the assumption on the numbering of the vertices, we know that $n$ is a sink vertex of $\Gamma_0$. It follows from Lemma~\ref{lem:easy_ineq_pd}(b) that $\pd_{\Lambda} M = \pd_{A_n} M_n$. Since $n$ is a sink vertex, the simple $k\Gamma/I$-module $S_n$ is projective, and thus it holds that $\pd_{\Lambda} M = \max\{\pd_{A_n} M_n,\pd_{k\Gamma/I} S_n\}$. This proves the initial step of induction. 

Now suppose that $\supp M \subseteq \{i,\ldots,n\}$ and that the statement is valid for representations whose support is contained in $\{i+1,\ldots,n\}$. Initially we are going to study the projective dimension of $M_i$ over $\Lambda$. If $i$ is a sink vertex, then, similarly to above, we have that $\pd_{\Lambda} M_i = \max\{\pd_{A_i} M_i,\pd_{k\Gamma/I} S_i\}$, so suppose $i$ is not a sink vertex. Let $(P,g)$ be a projective cover of $M_i$ over $A_i$. Then, because of Lemma~\ref{lem:main lemma}, there is an exact sequence in $\md \Lambda$:
\begin{displaymath}
\xymatrix{
0 \ar[r]& \Co_i(\Ker g) \oplus L \ar[r]& \Co_i(P) \ar[r]& M_i \ar[r]& 0}
\end{displaymath}
where $L$ satisfies the conditions given in the statement of the cited lemma. From this exact sequence, we deduce that
$$\pd_{\Lambda} M_i \leq \max\{\pd_{\Lambda} \Co_i (P),\pd_{\Lambda}(\Co_i(\Ker g) \oplus L)+1\}$$
Since $P$ is projective over $A_i$, Proposition~\ref{cone x projective} implies that $\pd_{\Lambda}\Co_i(P) = 0$. Thus
$$\pd_{\Lambda} M_i \leq \pd_{\Lambda}(\Co_i(\Ker g) \oplus L)+1 \leq \max \{\pd_{\Lambda} \Co_i(\Ker g), \pd_{\Lambda} L\}+1$$ 
Using Corollary~\ref{lem:pd x cone}, we have 
\begin{equation}
\label{eq:proj_dims}
\pd_{\Lambda} M_i \leq \max \{\pd_{A_i} \Ker g, \pd_{\Lambda} L\}+1 
\end{equation}
Now we divide our analysis in cases:\\
{\bf  Case 1:}  $\pd_{A_i} \Ker g \geq \pd_{\Lambda} L$. \\ 
In this case, Equation~\ref{eq:proj_dims} implies that $\pd_{\Lambda} M_i \leq \pd_{A_i} \Ker g +1 = \pd_{A_i} M_i$, because $(P,g)$ is the projective cover of $M_i$.\\
{\bf Case 2:} $\pd_{A_i} \Ker g \leq \pd_{\Lambda} L$ \\
Now, from Equation~\ref{eq:proj_dims}, $\pd_{\Lambda} M_i \leq \pd_{\Lambda} L +1$. In case $I \setminus \{1,\ldots,i\} = I \setminus \{1,\ldots,i-1\}$, from Lemma~\ref{lem:main lemma}, we get that $\pd_{\Lambda} L = 0$. Since we have already supposed in this case that $\pd_{A_i} \Ker g \leq \pd_{\Lambda} L$, then $\pd_{A_i} \Ker g = 0$. Again from Equation~\ref{eq:proj_dims}, $\pd_{\Lambda} M_i \leq 1$. Since $i$ is not a sink, we know that $S_i$ is not projective over $k\Lambda/I$ and so $\pd_{k\Lambda/I} S_i \geq 1$. Thus $\pd_{\Lambda} M_i \leq \pd_{k\Lambda/I} S_i$. \\
Assume now $I \setminus \{1,\ldots,i\} \neq I \setminus \{1,\ldots,i-1\}$. By Lemma~\ref{lem:main lemma}, $\pd_{A_j} L_j = 0$ for every $j$, and since the support of $L$ is contained in $\{i+1,\ldots,n\}$, by the induction hypothesis and because $\Lambda$ is terraced: 
$$\pd_{\Lambda} L \leq \max_{j \in \supp L}\{\pd_{k\Gamma/I} S_j\} \leq \pd_{k\Gamma/I} S_i - 1$$ 
Then $\pd_{\Lambda} M_i \leq \pd_{\Lambda} L +1 \leq \pd_{k\Gamma/I} S_i - 1 + 1 = \pd_{k\Gamma/I} S_i$. \\
Putting together all cases discussed above, we conclude that 
$$\pd_{\Lambda} M_i \leq \max \{\pd_{A_i}M_i,\pd_{k\Gamma/I} S_i\}$$
Now, using Proposition~\ref{prop:pd_submodule_quotient}, we have that 
$$\pd_{\Lambda} M \leq \max_{j \in \supp M} \pd_{\Lambda} M_j \leq \max_{j \in \supp M} \{\pd_{A_j}M_j,\pd_{k\Gamma/I} S_j\}$$
which proves the theorem.
\end{proof}

\begin{corollary}
\label{cor:gdrelationdependent}
Let $\Lambda = k(\Gamma,\Aa,I)$ be a terraced gbp-algebra. Then, for every $j \in \Gamma_0$, $\gd A_j \leq \gd \Lambda$, and the following inequality holds:

$$\gd \Lambda \leq \max_{j \in \Gamma_0}\left\{\gd \frac{k\Gamma}{I}, \gd A_j\right\}$$

\end{corollary}

\subsection{Opposite algebras} 
Before our next corollary, let us recall some facts concerning the opposite algebra of $\Lambda = k(\Gamma,\Aa, I)$. Denote by $\Gamma^{op}$ the quiver with the same vertices of $\Gamma$ and with reversed arrows. Also, $I^{op}$ will denote the  set of relations in $\Gamma^{op}$ obtained through inversion of the arrows in $I$. Finally, $\Aa^{op} = \{A_i^{op}:i \in \Gamma_0\}$ is the set where $A_i^{op}$ is the opposite algebra of $A_i$. With this notation, we have that $\Lambda^{op} \cong k(\Gamma^{op},\Aa^{op},I^{op})$. (See \cite{CC2}, Proposition 2. We shall refer to this in the next proof as {\it Fact I}).

Also, in a natural way, one can describe the representations of the opposite algebra in terms of the representations of the original one using the duality functor D = Hom$_k(-,k)$. Namely, if  $((M_i)_{i \in \Gamma_0},(\phi_{\alpha})_{\alpha \in \Gamma_1})$ is the representation of the $\Lambda$-module $M$, then the representation of the  $\Lambda^{op}$-module D$M$ is isomorphic to $(D(M_i)_{i \in \Gamma_0}, D(\phi_{\alpha})_{\alpha \in \Gamma_1})$. 
(See  \cite{CC2}, Proposition 3. We shall refer to this in the next proof as {\it Fact II}). As a consequence, we will have:

\begin{corollary}
\label{cor:calculate_every_id}

Let $\Lambda = k(\Gamma,\Aa,I)$ be a terraced  gbp-algebra, and let $M$ be a representation over $\Lambda$.  Then
$$\id_{\Lambda} M  = \max_{i \in \supp M} \{\id_{A_i} M_i, \id_{k\Gamma/I} S_i\}$$
where $S_i$ denotes the simple $k\Gamma/I$-module over the vertex $i$.
\end{corollary}

\begin{proof}
The idea is to use Theorem~\ref{th:pd_formula} and the fact that the duality functor anti-preserves homological properties. Again, let $D = \Hom(-,k)$ denote the duality functor. Let $S'_i$ denote the simple $k\Gamma^{op}/I^{op}$-module over the vertice $i$. Thus:
\begin{align*}
\id_{\Lambda} M &= \pd_{\Lambda^{op}} DM  &\text{(because }D\text{ is a duality)}\\
	&= \max_{i \in \supp M} \{\pd_{A_i^{op}} (DM)_i,\pd_{k\Gamma^{op}/I^{op}}S'_i\}  &\text{(Thm.~\ref{th:pd_formula} and Fact I)} \\
	&= \max_{i \in \supp M} \{\pd_{A_i^{op}} D(M_i), \pd_{k\Gamma^{op}/I^{op}}S'_i\}  &\text{(Fact II)} \\
	&= \max_{i \in \supp M} \{\pd_{A_i^{op}} D(M_i), \pd_{k\Gamma^{op}/I^{op}}D(S_i)\}  & \text{(because }D\text{ is a duality)}\\
	&= \max_{i \in \supp M} \{\id_{A_i} M_i,\id_{k\Gamma/I} S_i\}  &\text{(because }D\text{ is a duality)}
\end{align*}
\end{proof}

\subsection{Finitistic dimension} Given an algebra $A$, its  \textbf{finitistic dimension}\index{dimension!finitistic} is given by:
$$\fd A = \operatorname{sup} \{\pd_A M: M \text{ is an }A\text{-module of finite projective dimension}\}$$
A still open conjecture,  called the \textbf{Finitistic Dimension Conjecture},  states that every algebra has finite finitistic dimension.

\begin{proposition} \label{prop:findim}
Let $\Lambda = k(\Gamma,\Aa,I)$ be a terraced  gbp-algebra. Then 
$$\fd \Lambda \leq \max_{i \in \Gamma_0} \left\{\gd \frac{k\Gamma}{I}, \fd A_i \right\}$$
In particular, if the bound path algebra $k\Gamma/I$ has finite global dimension and $\fd A_i< \infty$ for each $i$,  then also $\fd \Lambda< \infty$.
\end{proposition}

\begin{proof}
Let $M = ((M_i)_{i \in \Gamma_0},(\phi_{\alpha})_{\alpha \in \Gamma_1})$ be a representation of finite projective dimension over $\Lambda$. From Lemma~\ref{lem:easy_ineq_pd}, for every $i \in \Gamma_0$, $\pd_{A_i} M_i \leq \pd_{\Lambda} M$, so $M_i$ has finite projective dimension over $A_i$, and thus $\pd_{A_i} M_i \leq \fd A_i$. Using Theorem~\ref{th:pd_formula}, 
$$\pd_{\Lambda} M\  \leq\  \max_{i \in \Gamma_0}\{\pd_{k\Gamma/I}S_i,\pd_{A_i}M_i\}\  \leq\  \max_{i \in \Gamma_0}\{\gd k\Gamma/I, \fd A_i\}. $$ 
Since $M$ is arbitrary, the statement follows.
\end{proof}

\section{Homological dimensions for gp-algebras}
\label{sec:some_consequences}

We shall now concentrate in gp-algebras which are, as observed above, terraced gbp-algebras. We start with the following result which is a direct consequence of the above considerations. 

\begin{theorem}
\label{cor:gd_formula}
Let $\Lambda = k(\Gamma,\Aa)$ be a  gp-algebra, with $\Gamma$ having at least one arrow. Then $\gd \Lambda = \max_{j \in \Gamma_0}\{ 1, \gd A_j\}$.

\end{theorem}
\begin{proof}
Observe that $\gd k\Gamma =1$ in this case and so, by Corollary \ref{cor:gdrelationdependent}, $\gd \Lambda \leq \max_{j \in \Gamma_0}\{ 1, \gd A_j\}$. The equality now follows using Corollary \ref{cor:easy_ineq_gd} and the fact that $\Lambda$ is not semisimple (since $k\Gamma$ is not).
\end{proof}

\subsection{Shod and quasitilted algebras} The next result is an application to the study of shod and quasitilted algebras. Quasitilted algebras were introduced in \cite{HRS} as a generalization of tilted algebras, by considering tilting objects in abelian categories. We shall, however, use a characterization of quasitilted algebras, also proven in \cite{HRS}, which suits better our purpose here. The shod algebras were then introduced in \cite{CLa1} in order to generalize the concept of quasitilted. The acronym shod stands for small homological dimension, as it is clear from the definition below. We refer to \cite{CLa1,HRS} for more details. 

\begin{definition}
Let $A$ be an algebra. We say that $A$ is a \textbf{shod} algebra if, for every indecomposable $A$-module $M$, either $\pd_A M \leq 1$ or $\id_A M \leq 1$. If, besides from being shod, $A$ has global dimension of at most two, we say that $A$ is \textbf{quasitilted}.
\end{definition}

Our next result allows us to produce a quasitilted or shod gp-algebra from other algebras. It is worth mentioning that it is not intended as a complete description of which generalized (bound) path algebras are quasitilted or shod. Before stating it, please note that every hereditary algebra is quasitilted, and thus also shod.

\begin{proposition}
Let $\Lambda = k(\Gamma, \Aa)$ be a gp-algebra, with $\Gamma$ acyclic. Suppose that $A_j$ is hereditary for every $j \in \Gamma_0$, except possibly for a single vertex $i \in \Gamma_0$. Then:

\begin{enumerate}
\item[(a)] If $A_i$ is shod, then $\Lambda$ is shod.
\item[(b)] If $A_i$ is quasitilted,  then $\Lambda$ is quasitilted.
\end{enumerate}
\end{proposition}

\begin{proof}
(a) Let $M = ((M_j)_{j \in \Gamma_0},(\phi_{\alpha})_{\alpha \in \Gamma_1})$ be an indecomposable representation over $\Lambda$. Since $\Gamma$ is acyclic, we infer that the  algebra $k\Gamma$ is hereditary and so every simple module over it will have projective and injective dimension of at most one. 
Observe also that, since  $A_j$ is hereditary for $j \neq i$, we also have $\pd_{A_j} M_j \leq 1 $ and $\id_{A_j} M_j \leq 1$ if $j \neq i$. \\
Now, since $A_i$ is shod, either $\pd_{A_i} M_i \leq 1$ or $\id_{A_i} M_i \leq 1$. In the former case, from Theorem~\ref{th:pd_formula}, we have that $\pd_{\Lambda} M \leq \max_{j \in \Gamma_0} \{\pd_{A_j} M_j, \pd_{k\Gamma}S_j\} \leq 1$, and in the latter, using Corollary~\ref{cor:calculate_every_id} in an analogous manner, one obtains that $\id_{\Lambda} M \leq 1$. Thus $\Lambda$ is shod. \\
(b) Since $A_i$ is quasitilted, it is shod and from the previous item we get that $\Lambda$ is shod. It remains to prove that $\gd \Lambda \leq 2$. Applying Corollary~\ref{cor:gdrelationdependent},
$$\gd \Lambda \leq \max_{j \in \Gamma_0} \{k\Gamma, \gd A_j\} \leq 2,$$ 
using that $A_i$ is quasitilted and that the other algebras are hereditary.
\end{proof}

\begin{example}
This example will show that the converse of proposition above could not hold. Let $A$ be the bound path algebra over the quiver
\begin{displaymath}
\xymatrix{
1 \ar[r]^{\alpha} & 2 \ar[r]^{\beta} & 3
}
\end{displaymath}
bound by $\alpha \beta = 0$, and let $\Lambda$ be the generalized path algebra given by
\begin{displaymath}
\xymatrix{
A \ar[rr] && k && A \ar[ll] 
}
\end{displaymath}
We have that, with this setting, $\Lambda$ does not satisfy the hypothesis from the last proposition: there is more than one vertex upon which the algebra is quasitilted and non-hereditary.
However, using \cite{ICNLP}, Theorem 3.3 or \cite{CC1}, Theorem 3.9, we see that $\Lambda$ is isomorphic to the bound path algebra over the quiver

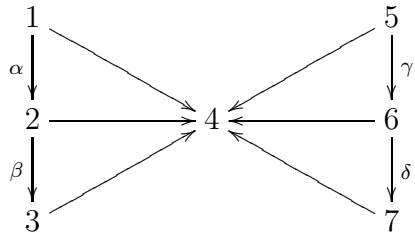
\begin{displaymath}
\xymatrix{
1 \ar[d]_{\alpha} \ar[drr] &&   && 5 \ar[d]^{\gamma} \ar[dll] \\
2 \ar[d]_{\beta} \ar[rr]  && 4  && 6 \ar[d]^{\delta} \ar[ll] \\
3 \ar[urr]        &&   && 7\ar[ull]
}
\end{displaymath}
bound by $\alpha \beta = \gamma \delta = 0$. Then it is easy to see that $\Lambda$ is a quasitilted algebra. The same example shows that the converse of the above proposition also does not hold for shod algebras.
\end{example}

We finish our considerations with a result which is a direct consequence of Proposition \ref{prop:findim}. 

\begin{proposition}
Let $\Lambda = k(\Gamma,\Aa)$ be a gp-algebra, with $\Gamma$ having at least one arrow. Then 
$$\fd \Lambda = \max_{i \in \Gamma_0} \left\{ 1, \fd A_i \right\}$$
In particular, if  $\fd A_i< \infty$ for each $i$,  then also $\fd \Lambda< \infty$.
\end{proposition}

\begin{proof}
Just observe that $\gd k\Gamma = 1$, and use Proposition \ref{prop:findim}. 
\end{proof}

\subsection*{Acknowledgements}
 The authors gratefully acknowledge financial support by São Paulo Research Foundation (FAPESP), grants \#2018/18123-5, \#2020/13925-6 and  \#2022/02403-4. The second author has also a grant by CNPq (Pq 312590/2020-2).

\end{document}